\documentclass[12pt,reqno]{amsart}
\usepackage{amsmath, amsfonts, amssymb, amsthm, hyperref}
\usepackage{bm}
\allowdisplaybreaks[4]
\def\l{\left}
\def\r{\right}
\def\bg{\bigg}
\def\({\bg(}
\def\){\bg)}
\def\t{\text}
\def\f{\frac}

\def\eq{\equiv}

\def\Z{\mathbb Z}
\def\C{\mathbb C}
\def\N{\mathbb N}
\def\Q{\mathbb Q}

\def\1{{\bf 1}}

\theoremstyle{plain}
\newtheorem{theorem}{Theorem}[section]
\newtheorem{lemma}{Lemma}

\newtheorem{conjecture}{Conjecture}

\theoremstyle{definition}

\theoremstyle{remark}
\newtheorem{remark}{Remark}

\def\<{\langle}
\def\>{\rangle}

\begin{document}
\hbox{}
\medskip

\title[$p$-adic analogues of hypergeometric identities]{$p$-adic analogues of hypergeometric identities\\and their applications}

\author{Chen Wang}
\address {(Chen Wang) Department of Mathematics, Nanjing
University, Nanjing 210093, People's Republic of China}
\email{cwang@smail.nju.edu.cn}

\author{Zhi-Wei Sun}
\address {(Zhi-Wei Sun) Department of Mathematics, Nanjing
University, Nanjing 210093, People's Republic of China}
\email{zwsun@nju.edu.cn}

\thanks{This work was supported by the National Natural Science Foundation of China (grant no. 11971222)}
\begin{abstract} In this paper, we confirm several conjectures posed by Sun recently. For example, we prove that for any odd prime $p$ we have
$$
\sum_{k=0}^{p-1}A_k\equiv\begin{cases}4x^2-2p\pmod{p^2}\quad&\text{if $p=x^2+2y^2\ (x,y\in\mathbb{Z})$},\\ 0\pmod{p^2}\quad&\text{if $p\equiv5,7\pmod{8}$},\end{cases}
$$
where $A_n:=\sum_{k=0}^n\binom{n+k}{k}^2\binom{n}{k}^2$ are the Ap\'{e}ry numbers.
\end{abstract}
\maketitle

\section{Introduction}
\setcounter{lemma}{0}
\setcounter{theorem}{0}
\setcounter{equation}{0}
\setcounter{conjecture}{0}
\setcounter{remark}{0}

For $n\in\N=\{0,1,2,\ldots\}$ the truncated hypergeometric series ${}_{r+1}F_r$ are defined by
$$
{}_{r+1}F_r\bigg[\begin{matrix}\alpha_0&\alpha_1&\cdots&\alpha_r\\ &\beta_1&\cdots&\beta_r\end{matrix}\bigg|\ z\bigg]_n:=\sum_{k=0}^{n}\f{(\alpha_0)_k\cdots(\alpha_r)_k}{(\beta_1)_k\cdots(\beta_r)_k}\cdot\f{z^k}{k!},
$$
where $\alpha_0,\ldots,\alpha_r,\beta_1,\ldots,\beta_r,z\in\C$ and
$$
(\alpha)_k:=
\begin{cases}
\displaystyle \prod_{j=0}^{k-1}(\alpha+j),\quad&\t{if $k\geq1$},\vspace{1mm}\\
\displaystyle 1,\quad&\t{if $k=0$},
\end{cases}
$$
denotes the so-called Pochhammer's symbol. Clearly, the truncated hypergeometric series is the sum of the first finite terms of the corresponding hypergeometric series. In the past decades, the arithmetic properties of the truncated hypergeometric series have been widely studied (cf. \cite{AhOn00,DFLST16,He15,He17a,He17b,Liu17,Liu19,Long11,LoRa16,MP,Mo04,Mo05,PTW,SunZW11a,SunZW11b,SunZW13,Swisher15,Ta12,vH,Wang}).

The well-known Ap\'{e}ry numbers given by
$$
A_n:=\sum_{k=0}^n\binom{n+k}{k}^2\binom{n}{k}^2=\sum_{k=0}^n\binom{n+k}{2k}^2\binom{2k}{k}^2\quad (n\in\N=\{0,1,\ldots\}),
$$
were first introduced by Ap\'{e}ry to prove the irrationality of $\zeta(3)=\sum_{n=1}^{\infty}1/n^3$ (see \cite{Apery,Sl}). In 2012, Sun \cite{Sunapery} studied the sums involving Ap\'{e}ry numbers systematically and posed some conjectures; for example, he conjectured that for any odd prime $p$ we have
\begin{equation}\label{apery1}
\sum_{k=0}^{p-1}A_k\equiv\begin{cases}4x^2-2p\pmod{p^2}\quad&\text{if $p=x^2+2y^2\ (x,y\in\mathbb{Z})$},\\ 0\pmod{p^2}\quad&\text{if $p\equiv5,7\pmod{8}$}.\end{cases}
\end{equation}
Note that the above conjecture was also collected in \cite[Conjecture 55]{SunZW19}. We now state our first theorem.
\begin{theorem}\label{aperyth}
For any odd prime $p$, \eqref{apery1} holds.
\end{theorem}
\begin{remark}
In \cite{Sunapery}, Sun proved that \eqref{apery1} holds modulo $p$.
\end{remark}

Recently, Sun \cite[Conjectures 35 and 36]{SunZW19} proposed a series of congruences involving the following polynomial in $x$:
$$
\sum_{k=0}^{n-1}\varepsilon^k(2k+1)^{2l-1}\sum_{j=0}^k\binom{-x}{j}^m\binom{x-1}{k-j}^m,
$$
where $\varepsilon\in\{\pm1\}$ and $n,l,m\in\Z^+=\{1,2,3,\ldots\}$. He conjectured that the polynomial
\begin{equation}\label{sunintval1}
\f{1}{n}\sum_{k=0}^{n-1}\varepsilon^k(2k+1)^{2l-1}\sum_{j=0}^k\binom{-x}{j}^m\binom{x-1}{k-j}^m
\end{equation}
is integer-valued, here we say a polynomial $P(x)\in\Q[x]$ is integer-valued if $P(m)\in\Z$ for all $m\in\Z$. If $\varepsilon=1$ and $m=2$, he even conjectured that the polynomial
\begin{equation}\label{sunintval2}
\f{(2l-1)!!}{n^2}\sum_{k=0}^{n-1}(2k+1)^{2l-1}\sum_{j=0}^k\binom{-x}{j}^2\binom{x-1}{k-j}^2
\end{equation}
is integer-valued. Sun also posed some conjectures when $n$ take prime value. Let us consider the case that $l=1,m\geq3$ and $n=p$ is an odd prime. Exchanging the summation order and replacing $k-j$ with $k$ we obtain
$$
\sum_{k=0}^{p-1}\varepsilon^k(2k+1)\sum_{j=0}^k\binom{-x}{j}^m\binom{x-1}{k-j}^m=\sum_{j=0}^{p-1}\varepsilon^j\binom{-x}{j}^m\sum_{k=0}^{p-1-j}\varepsilon^k(2k+2j+1)\binom{x-1}{k}^m.
$$
Denote by $\<-x\>_p$ the least nonnegative residue of $-x$ modulo $p$. Clearly,
$$\binom{-x}{j}\eq0\pmod{p}\quad\t{for}\quad j\in\{\<-x\>_p+1,\ldots,p-1\}$$
and
$$\binom{x-1}{k}\eq0\pmod{p}\quad\t{for}\quad k\in\{p-\<-x\>_p,\ldots,p-1\}.$$
Therefore by noting that $p-1-j\geq p-1-\<-x\>_p$ for any $j\in\{0,\ldots,a\}$ we have
\begin{equation}\label{sigma1sigma2}
\begin{aligned}
\sum_{k=0}^{p-1}\varepsilon^k(2k+1)\sum_{j=0}^k\binom{-x}{j}^m\binom{x-1}{k-j}^m\eq&\sum_{j=0}^{p-1}\varepsilon^j\binom{-x}{j}^m\sum_{k=0}^{p-1}\varepsilon^k(2k+2j+1)\binom{x-1}{k}^m\\
=&(1-x)\Sigma_1+x\Sigma_2\pmod{p^m},
\end{aligned}
\end{equation}
where
\begin{align*}
\Sigma_1:=&{}_{m+1}F_m\bigg[\begin{matrix}1-x&1+\f{1-x}{2}&1-x&\cdots&1-x\\ &\f{1-x}{2}&1&\cdots&1\end{matrix}\bigg|\ (-1)^m\varepsilon\bigg]_{p-1}\\
&\times {}_{m}F_{m-1}\bigg[\begin{matrix}x&x&\cdots&x\\ &1&\cdots&1\end{matrix}\bigg|\ (-1)^m\varepsilon\bigg]_{p-1}
\end{align*}
and
\begin{align*}
\Sigma_2:=&{}_{m+1}F_m\bigg[\begin{matrix}x&1+\f{x}{2}&x&\cdots&x\\ &\f{x}{2}&1&\cdots&1\end{matrix}\bigg|\ (-1)^m\varepsilon\bigg]_{p-1}\\
&\times {}_{m}F_{m-1}\bigg[\begin{matrix}1-x&1-x&\cdots&1-x\\ &1&\cdots&1\end{matrix}\bigg|\ (-1)^m\varepsilon\bigg]_{p-1}.
\end{align*}
In view of the above we only need to consider the congruences concerning the truncated ${}_{m+1}F_m$ and ${}_mF_{m-1}$ hypergeometric series. This is the motivation of the remaining part of this paper.

Our results involve the Morita's $p$-adic gamma function $\Gamma_p$ (cf. \cite{Robert00}) which is the $p$-adic analogue of the classic gamma function $\Gamma$. For $n\in\N$ define $\Gamma_p(0):=1$ and for $n\geq1$
$$
\Gamma_p(n):=(-1)^n\prod_{\substack{1\leq k<n\\ p\nmid k}}k.
$$
As we all know, the definition of $\Gamma_p$ can be extended to $\Z_p$ since $\N$ is a dense subset of $\Z_p$ with respect to $p$-adic norm, where $\Z_p$ denotes the ring of $p$-adic integers. It follows that
\begin{equation}\label{padicgamma1}
\f{\Gamma_p(x+1)}{\Gamma_p(x)}=\begin{cases}\displaystyle -x,\quad&\t{if $p\nmid x$},\vspace{2mm}\\ \displaystyle -1,\quad&\t{if $p\mid x$}.\end{cases}
\end{equation}
For more properties of $p$-adic gamma functions, one may consult \cite{LoRa16,MP,PTW,Robert00}.

We now state our second theorem.
\begin{theorem}\label{4F3analog}
For any odd prime $p$ and $\alpha\in\Z_p^{\times}=\{x\in\Z_p\ |\ p\nmid x\}$. Let $s=(\alpha+\<-\alpha\>_p)/p$ and
$$
h_p(\alpha)=\f{\Gamma_p\l(\f{1+\alpha}{2}\r)\Gamma_p\l(\f{1-3\alpha}{2}\r)}{\Gamma_p(1+\alpha)\Gamma_p(1-\alpha)\Gamma_p\l(\f{1-\alpha}{2}\r)^2}.
$$
Then the following congruence holds modulo $p^3$,
$$
{}_4F_3\bigg[\begin{matrix}\alpha&1+\f{\alpha}{2}&\alpha&\alpha\\ &\f{\alpha}{2}&1&1\end{matrix}\bigg|\ 1\bigg]_{p-1}\eq
\begin{cases}
\displaystyle 2h_p(\alpha),\quad & \t{if $\<-\alpha\>_p$ is odd and $\<-\alpha\>_p<\f{2p+1}{3}$,}\vspace{2mm}\\
\displaystyle (2-3s)ph_p(\alpha),\quad & \t{if $\<-\alpha\>_p$ is odd and $\<-\alpha\>_p\geq\f{2p+1}{3}$,}\vspace{2mm}\\
\displaystyle sph_p(\alpha),\quad & \t{if $\<-\alpha\>_p$ is even and $\<-\alpha\>_p<\f{p+1}{3}$,}\vspace{2mm}\\
\displaystyle\f{(s-3s^2)p^2h_p(\alpha)}{2},\quad & \t{if $\<-\alpha\>_p$ is even and $\<-\alpha\>_p\geq\f{p+1}{3}$.}
\end{cases}
$$
\end{theorem}
\begin{remark}
In 2017, He \cite{He17b} studied the congruences modulo $p^2$ for primes $p\geq5$ and $\alpha=1/2,1/3,1/4$.
\end{remark}

In \cite{MP}, Mao and Pan obtained a number of congruences modulo $p^2$ involving truncated hypergeometric identities and $p$-adic gamma functions. For instance, as a corollary, they proved that for any odd prime $p$ and $\alpha,\beta\in\Z_p$ with $\<-\alpha\>_p\leq\<-\beta\>_p\leq(p+\<-\alpha\>_p-1)/2$ and $(\alpha-\beta+1)_{p-1}\not\eq0\pmod{p^2}$
$$
{}_4F_3\bigg[\begin{matrix}\alpha&1+\f{\alpha}{2}&\alpha&\beta\\ &\f{\alpha}{2}&1&\alpha-\beta+1\end{matrix}\bigg|-1\bigg]_{p-1}\eq-(\alpha+\<-\alpha\>_p)\cdot\f{\Gamma_p(\alpha-\beta+1)}{\Gamma_p(1+\alpha)\Gamma_p(1-\beta)}\pmod{p^2}.
$$
Letting $\beta=\alpha$ in the above congruence we get that
\begin{equation}\label{4F3(-1)modp2}
{}_4F_3\bigg[\begin{matrix}\alpha&1+\f{\alpha}{2}&\alpha&\alpha\\ &\f{\alpha}{2}&1&1\end{matrix}\bigg|-1\bigg]_{p-1}\eq\f{\alpha+\<-\alpha\>_p}{\Gamma_p(1+\alpha)\Gamma_p(1-\alpha)}\pmod{p^2}.
\end{equation}

Our next theorem says that \eqref{4F3(-1)modp2} is also valid for modulo $p^3$.
\begin{theorem}\label{4F3(-1)analog}
Let $p$ be an odd prime and $\alpha\in\Z_p^{\times}$. Then we have
$$
{}_4F_3\bigg[\begin{matrix}\alpha&1+\f{\alpha}{2}&\alpha&\alpha\\ &\f{\alpha}{2}&1&1\end{matrix}\bigg|-1\bigg]_{p-1}\eq\f{\alpha+\<-\alpha\>_p}{\Gamma_p(1+\alpha)\Gamma_p(1-\alpha)}\pmod{p^3}.
$$
\end{theorem}
\begin{remark}
Note that the cases $\alpha=1/d$ and $p\eq1\pmod{d}$ with $d=2,3,4$ were first conjectured by van Hamme \cite{vH} and confirmed by Swisher \cite{Swisher15}.
\end{remark}

The proofs of Theorems \ref{4F3analog} and \ref{4F3(-1)analog} depend on the local-global theorem for $p$-adic supercongruences established by Pan, Tauraso and Wang \cite{PTW}. Here we illustrate the local-global theorem briefly (the reader may refer to \cite[Theorem 1.1]{PTW} for details). For any prime $p>\binom{r+1}{2}$ the local-global theorem says that if a congruence modulo $p^r$ holds over some $r$ admissible hyperplanes of $\Z_p^n$, then it also holds over the whole $\Z_p^n$. In view of this, to show our theorems, we only need to prove them `locally'.

The organization of this paper is as follows. In the next section, we shall prove Theorem \ref{aperyth}. We will prove Theorems \ref{4F3analog} and \ref{4F3(-1)analog} in Section 3. In Section 4, we shall confirm some conjectures of Sun in \cite[Conjectures 35 and 36]{SunZW19} as applications of Theorems \ref{4F3analog} and \ref{4F3(-1)analog}. In the last section, we will prove more conjectures of Z.-W. Sun by some known results.

\medskip
\section{Proof of Theorem \ref{aperyth}}
\setcounter{lemma}{0}
\setcounter{theorem}{0}
\setcounter{equation}{0}
\setcounter{conjecture}{0}
\setcounter{remark}{0}
In order to show Theorem \ref{aperyth} we need the following lemmas.
\begin{lemma}\cite[(2.5)]{Guo}\label{guoid}
We have the following identity
$$
\binom{k}{i}\binom{k+i}{i}\binom{k}{j}\binom{k+j}{j}=\sum_{s=\max\{i,j\}}^{i+j}\binom{s}{i}\binom{s}{j}\binom{i+j}{s}\binom{k}{s}\binom{k+s}{s}.
$$
\end{lemma}

The following identity can be verified by induction on $j$.
\begin{lemma}\label{sj2j}For $j\in\N$ we have
$$
\sum_{s=j}^{2j}\binom{s}{j}^2\binom{2j}{s}\f{(-1)^s}{2s+1}=\f{\binom{2j}{j}^2}{(4j+1)\binom{4j}{2j}}.
$$
\end{lemma}

\begin{lemma}\cite[Theorem 3.5.5]{AAR99}\label{3F2trans}
If $a+b=1$ and $e+f=2c+1$, then we have
$$
{}_3F_2\bigg[\begin{matrix}a&b&c\\&e&f\end{matrix}\bigg|\ 1\bigg]=\f{\pi\Gamma(e)\Gamma(f)}{2^{2c-1}\Gamma((a+e)/2)\Gamma((a+f)/2)\Gamma((b+e)/2)\Gamma((b+f)/2)}.
$$
\end{lemma}
The classical gamma function has the following Gauss multiplication formula \cite[Page 371]{Robert00}.
\begin{lemma}\label{Gauss}
For $z\in\C$ and $m\geq2$ we have
$$
\prod_{0\leq j<m}\Gamma\l(z+\f{j}{m}\r)=(2\pi)^{(m-1)/2}m^{(1-2mz)/2}\cdot\Gamma(mz).
$$
\end{lemma}
In fact, we prove the following result.
\begin{theorem}\label{padicgamma}For any odd prime $p$ we have
$$
\sum_{k=0}^{p-1}A_k\eq\begin{cases}-\l(\f{-1}{p}\r)\Gamma_p\l(\f{1}{8}\r)^2\Gamma_p\l(\f{3}{8}\r)^2\pmod{p^2}\quad&\text{if $p\eq1,3\pmod{8}$},\\ 0\pmod{p^2}\quad&\text{if $p\equiv5,7\pmod{8}$}.\end{cases}
$$
\end{theorem}
\begin{remark}
(a) Here we illustrate that Theorem \ref{padicgamma} implies Theorem \ref{aperyth}. By \cite{Mo05} and \cite{SunZW12}, if $p\eq1,3\pmod{8}$ and $p=x^2+2y^2\ (x,y\in\mathbb{Z})$ we have
$$
{}_4F_3\bigg[\begin{matrix}\f12&\f14&\f34\\&1&1\end{matrix}\bigg|\ 1\bigg]_{p-1}\eq4x^2-2p\pmod{p^2}.
$$
From \cite{PTW} we have for $p\eq1,3\pmod{8}$
$$
{}_4F_3\bigg[\begin{matrix}\f12&\f14&\f34\\&1&1\end{matrix}\bigg|\ 1\bigg]_{p-1}\eq-\l(\f{-1}{p}\r)\Gamma_p\l(\f{1}{8}\r)^2\Gamma_p\l(\f{3}{8}\r)^2\pmod{p^2}.
$$
Combining the above, for $p\eq1,3\pmod{8}$ we have
$$
4x^2-2p\eq-\l(\f{-1}{p}\r)\Gamma_p\l(\f{1}{8}\r)^2\Gamma_p\l(\f{3}{8}\r)^2\pmod{p^2}.
$$
Thus Theorem \ref{aperyth} holds.

(b) In \cite[Conjecture 55]{SunZW19}, Sun also conjectured that if prime $p>3$ and $p\eq1,3\pmod{8}$, then
\begin{equation}\label{apery3F2}
\sum_{k=0}^{p-1}A_k\eq{}_3F_2\bigg[\begin{matrix}\f12&\f14&\f34\\&1&1\end{matrix}\bigg|\ 1\bigg]_{p-1}\pmod{p^3}.
\end{equation}
Actually, Theorem \ref{padicgamma} implies that \eqref{apery3F2} holds modulo $p^2$.
\end{remark}

\medskip
\noindent{\it Proof of Theorem \ref{padicgamma}}. By Lemma \ref{guoid} it is easy to see that
\begin{align*}
\sum_{k=0}^{p-1}A_k=&\sum_{k=0}^{p-1}\sum_{j=0}^{k}\binom{k+j}{j}^2\binom{k}{j}^2\\
=&\sum_{k=0}^{p-1}\sum_{j=0}^k\sum_{s=j}^{2j}\binom{s}{j}^2\binom{2j}{s}\binom{k}{s}\binom{k+s}{s}\\
=&\sum_{j=0}^{p-1}\sum_{s=j}^{2j}\binom{s}{j}^2\binom{2j}{s}\sum_{k=j}^{p-1}\binom{k+s}{s}\binom{k}{s}\\
=&\sum_{j=0}^{p-1}\sum_{s=j}^{2j}\binom{s}{j}^2\binom{2j}{s}\f{p}{2s+1}\binom{p+s}{s}\binom{p-1}{s},
\end{align*}
where in the last step we use the following identity which could be checked by induction on $n$:
$$
\sum_{k=s}^{n-1}\binom{k+s}{s}\binom{k}{s}=\f{n}{2s+1}\binom{n+s}{s}\binom{n-1}{s}.
$$
Note that for $s\in\{0,\ldots,p-1\}$ we have $ord_p(2s+1)\leq1$ and
$$
\binom{p+s}{s}\binom{p-1}{s}\eq(-1)^s\pmod{p^2}.
$$
Therefore, by Lemma \ref{sj2j} we have
\begin{align*}
\sum_{k=0}^{p-1}A_k\eq&p\sum_{j=0}^{p-1}\sum_{s=j}^{\min\{p-1,2j\}}\binom{s}{j}^2\binom{2j}{s}\f{(-1)^s}{2s+1}\\
=&p\sum_{j=0}^{(p-1)/2}\sum_{s=j}^{2j}\binom{s}{j}^2\binom{2j}{s}\f{(-1)^s}{2s+1}+p\sum_{j=(p+1)/2}^{p-1}\sum_{s=j}^{p-1}\binom{s}{j}^2\binom{2j}{s}\f{(-1)^s}{2s+1}\\
=&p\sum_{j=0}^{(p-1)/2}\f{\binom{2j}{j}^2}{(4j+1)\binom{4j}{2j}}+p\sum_{j=(p+1)/2}^{p-1}\sum_{s=j}^{p-1}\binom{s}{j}^2\binom{2j}{s}\f{(-1)^s}{2s+1}\\
=&p\cdot{}_3F_2\bigg[\begin{matrix}\f12&\f12&\f12\\&\f34&\f54\end{matrix}\bigg|\ 1\bigg]_{(p-1)/2}+p\sum_{j=(p+1)/2}^{p-1}\sum_{s=j}^{p-1}\binom{s}{j}^2\binom{2j}{s}\f{(-1)^s}{2s+1}\pmod{p^2}.
\end{align*}
It is easy to see that $ord_p(2s+1)=0$ and
$$
\binom{2j}{s}\binom{s}{j}=\binom{2j}{j}\binom{j}{s-j}\eq0\pmod{p}
$$
provided that $s,j\in\{(p+1)/2,\ldots,p-1\}$. Hence we have
$$
\sum_{k=0}^{p-1}A_k\eq p\cdot{}_3F_2\bigg[\begin{matrix}\f12&\f12&\f12\\&\f34&\f54\end{matrix}\bigg|\ 1\bigg]_{(p-1)/2}\pmod{p^2}.
$$
Clearly, at most one of $\<-3/4\>_p$ and $\<-5/4\>_p$ is smaller than $(p-1)/2$. Thus for $k\leq(p-1)/2$ we have
$$
\f{p}{(1)_k\l(\f{3}{4}\r)_k\l(\f{5}{4}\r)_k}\in\Z_p.
$$
By Lemmas \ref{3F2trans} and \ref{Gauss} and noting that $((1+p)/2)_k((1-p)/2)_k\eq(1/2)_k^2\pmod{p^2}$ we further get that
\begin{align*}
\sum_{k=0}^{p-1}A_k\eq& p\cdot{}_3F_2\bigg[\begin{matrix}\f12&\f12&\f12\\&\f34&\f54\end{matrix}\bigg|\ 1\bigg]_{\f{p-1}2}\\
\eq&p\cdot{}_3F_2\bigg[\begin{matrix}\f{1-p}2&\f{1+p}2&\f12\\&\f34&\f54\end{matrix}\bigg|\ 1\bigg]_{\f{p-1}2}\\
=&\f{p\Gamma\l(\f34\r)\Gamma\l(\f54\r)\Gamma\l(\f12\r)^2}{\Gamma\l(\f{5-2p}{8}\r)\Gamma\l(\f{7-2p}{8}\r)\Gamma\l(\f{5+2p}{8}\r)\Gamma\l(\f{7+2p}{8}\r)}\\
=&\f{p\Gamma\l(\f38\r)\Gamma\l(\f58\r)\Gamma\l(\f78\r)\Gamma\l(\f98\r)}{\Gamma\l(\f{5-2p}{8}\r)\Gamma\l(\f{7-2p}{8}\r)\Gamma\l(\f{5+2p}{8}\r)\Gamma\l(\f{7+2p}{8}\r)}\pmod{p^2}.
\end{align*}

If $p\eq1\pmod{8}$. In this case, we have
\begin{gather*}
\f{\Gamma\l(\f38\r)}{\Gamma\l(\f{5-2p}{8}\r)}=\f{\Gamma_p\l(\f38\r)}{\Gamma_p\l(\f{5-2p}{8}\r)}\cdot(-1)^{(p-1)/4},\quad \f{\Gamma\l(\f58\r)}{\Gamma\l(\f{7-2p}{8}\r)}=\f{\Gamma_p\l(\f58\r)}{\Gamma_p\l(\f{7-2p}{8}\r)}\cdot(-1)^{(p-1)/4},\\
\f{\Gamma\l(\f78\r)}{\Gamma\l(\f{5+2p}{8}\r)}=\f{\Gamma_p\l(\f78\r)}{\Gamma_p\l(\f{5+2p}{8}\r)}\cdot(-1)^{(p-1)/4},\quad
\f{\Gamma\l(\f98\r)}{\Gamma\l(\f{7+2p}{8}\r)}=\f{\Gamma_p\l(\f98\r)}{\f{p}{8}\cdot\Gamma_p\l(\f{7+2p}{8}\r)}\cdot(-1)^{(p-1)/4}.
\end{gather*}
It follows that
\begin{align*}
\sum_{k=0}^{p-1}A_k\eq&\f{8\Gamma_p\l(\f38\r)\Gamma_p\l(\f58\r)\Gamma_p\l(\f78\r)\Gamma_p\l(\f98\r)}{\Gamma_p\l(\f{5-2p}{8}\r)\Gamma_p\l(\f{7-2p}{8}\r)\Gamma_p\l(\f{5+2p}{8}\r)\Gamma_p\l(\f{7+2p}{8}\r)}\\
\eq&-\f{\Gamma_p\l(\f38\r)\Gamma_p\l(\f18\r)}{\Gamma_p\l(\f58\r)\Gamma_p\l(\f78\r)}\pmod{p^2}.
\end{align*}
It is known that
$$
\Gamma_p(z)\Gamma_p(1-z)=(-1)^{p-\<-z\>_p}
$$
for $z\in\Z_p$ (cf. \cite{Robert00}). Thus
$$
\f{1}{\Gamma_p\l(\f58\r)\Gamma_p\l(\f78\r)}=\Gamma_p\l(\f38\r)\Gamma_p\l(\f18\r)(-1)^{2p-(5p-5)/8-(7p-7)/8}=\Gamma_p\l(\f38\r)\Gamma_p\l(\f18\r).
$$
Therefore
$$
\sum_{k=0}^{p-1}A_k\eq-\Gamma_p\l(\f{1}{8}\r)^2\Gamma_p\l(\f{3}{8}\r)^2\pmod{p^2}.
$$

If $p\eq3\pmod{8}$. Similarly, we arrive at
\begin{gather*}
\f{\Gamma\l(\f38\r)}{\Gamma\l(\f{5+2p}{8}\r)}=\f{\Gamma_p\l(\f38\r)}{\f{p}{8}\cdot\Gamma_p\l(\f{5+2p}{8}\r)}\cdot(-1)^{(p+1)/4},\quad \f{\Gamma\l(\f58\r)}{\Gamma\l(\f{7+2p}{8}\r)}=\f{\Gamma_p\l(\f58\r)}{\Gamma_p\l(\f{7+2p}{8}\r)}\cdot(-1)^{(p+1)/4},\\
\f{\Gamma\l(\f78\r)}{\Gamma\l(\f{5-2p}{8}\r)}=\f{\Gamma_p\l(\f78\r)}{\Gamma_p\l(\f{5-2p}{8}\r)}\cdot(-1)^{(p+1)/4},\quad
\f{\Gamma\l(\f98\r)}{\Gamma\l(\f{7-2p}{8}\r)}=\f{\Gamma_p\l(\f98\r)}{\Gamma_p\l(\f{7-2p}{8}\r)}\cdot(-1)^{(p+1)/4}.
\end{gather*}
Thus we also obtain that
$$
\sum_{k=0}^{p-1}A_k\eq-\f{\Gamma_p\l(\f38\r)\Gamma_p\l(\f18\r)}{\Gamma_p\l(\f58\r)\Gamma_p\l(\f78\r)}\pmod{p^2}.
$$
However, in this case we have
$$
\f{1}{\Gamma_p\l(\f58\r)\Gamma_p\l(\f78\r)}=\Gamma_p\l(\f38\r)\Gamma_p\l(\f18\r)(-1)^{2p-(7p-5)/8-(5p-7)/8}=-\Gamma_p\l(\f38\r)\Gamma_p\l(\f18\r).
$$
It follows that
$$
\sum_{k=0}^{p-1}A_k\eq\Gamma_p\l(\f{1}{8}\r)^2\Gamma_p\l(\f{3}{8}\r)^2\pmod{p^2}.
$$
The remaining cases can be proved similarly. The proof of Theorem \ref{padicgamma} is now complete.\qed

\medskip
\section{Proofs of Theorems \ref{4F3analog} and \ref{4F3(-1)analog}}
\setcounter{lemma}{0}
\setcounter{theorem}{0}
\setcounter{equation}{0}
\setcounter{conjecture}{0}
\setcounter{remark}{0}
Theorem \ref{4F3analog} is actually a $p$-aidc analogue of the following ${}_4F_3$ identity.
\begin{lemma}\cite[Page 182, 25(a)]{AAR99}\label{4F3identity} For any $\alpha,\beta,\gamma\in\C$ we have
\begin{align}
&{}_4F_3\bigg[\begin{matrix}\alpha&1+\f{\alpha}{2}&\beta&\gamma\\ &\f{\alpha}{2}&1+\alpha-\beta&1+\alpha-\gamma\end{matrix}\bigg|\ 1\bigg]\notag\\
=&\f{\Gamma(1+\alpha-\beta)\Gamma(1+\alpha-\gamma)\Gamma\l(\f{1+\alpha}{2}\r)\Gamma\l(\f{1+\alpha}{2}-\beta-\gamma\r)}{\Gamma(1+\alpha)\Gamma(1+\alpha-\beta-\gamma)\Gamma\l(\f{1+\alpha}{2}-\beta\r)\Gamma(\f{1+\alpha}{2}-\gamma)}.
\end{align}
\end{lemma}

The following lemma gives the well-known Euler's reflection formula and its $p$-adic analogue.
\begin{lemma}\cite[Pages 369--371]{Robert00}
For any $z\in\C$,
\begin{equation}\label{eulerreflection}
\Gamma(z)\Gamma(1-z)=\f{\pi}{\sin\pi z}.
\end{equation}
The above formula has a $p$-adic analogue as follows
\begin{equation}\label{eulerpadic}
\Gamma_p(z)\Gamma_p(1-z)=(-1)^{p-\<-z\>_p}
\end{equation}
for $z\in\Z_p$.
\end{lemma}

\medskip
\noindent{\it Proof of Theorem \ref{4F3analog}}. Here we just give the detailed proof of the first case where $\<-\alpha\>_p$ is odd and $\<-\alpha\>_p<\f{2p+1}{3}$, since the proofs of other cases are similar.

Let $a=\<-\alpha\>_p$. Now we assume that $a$ is even and $a<\f{2p+1}{3}$. Set
\begin{align*}
\Psi(x,y,z):=&{}_4F_3\bigg[\begin{matrix}-a+x&1+\f{-a+x}{2}&-a+y&-a+z\\ &\f{-a+x}{2}&1+x-y&1+x-z\end{matrix}\bigg|\ 1\bigg]_{p-1}\notag\\
&-\f{2\Gamma_p(1+x-y)\Gamma_p(1+x-z)\Gamma_p\l(\f{1-a+x}{2}\r)\Gamma_p\l(\f{1+3a+x-2y-2z}{2}\r)}{\Gamma_p(1-a+x)\Gamma_p(1+a+x-y-z)\Gamma_p\l(\f{1+a+x-2y}{2}\r)\Gamma_p\l(\f{1+a+x-2z}{2}\r)}.
\end{align*}
It is easy to see that
$$
{}_4F_3\bigg[\begin{matrix}\alpha&1+\f{\alpha}{2}&\alpha&\alpha\\ &\f{\alpha}{2}&1&1\end{matrix}\bigg|\ 1\bigg]_{p-1}\eq2g_p(\alpha)\pmod{p^3}
$$
is equivalent to
\begin{equation}\label{key}
\Psi(sp,sp,sp)\eq0\pmod{p^3},
\end{equation}
where $s=(\alpha+a)/p$. For $p=3,5$ we can verify \eqref{key} for any $1\leq\alpha\leq p^3$ numerically. Now we assume that $p\geq7$. In view of the local-global theorem from \cite{PTW}, we only need to show that
\begin{equation}\label{key'}
\Psi(rp,sp,tp)\eq0\pmod{p^3}
\end{equation}
provided that $r,s,t\in\Z_p$ and at least one of $r,s,t$ is zero. We first show that
\begin{equation}\label{r=0}
\Psi(0,sp,tp)=0
\end{equation}
for each $s,t\in\Z_p$. In fact, we may assume that $sp,tp,(s+t)p\in\Q\setminus\Z$ (since any $x\in\Z\cap\Z_p$ can be approximated by a sequence of $p$-adic integers $\{x_m\}_{m\geq0}$ in $(\Q\setminus\Z)\cap\Z_p$). By \eqref{4F3identity} we have
\begin{align*}
&{}_4F_3\bigg[\begin{matrix}-a&1+\f{-a}{2}&-a+sp&-a+tp\\ &\f{-a}{2}&1-sp&1-tp\end{matrix}\bigg|\ 1\bigg]_{p-1}\\
=&\lim_{z\to0}{}_4F_3\bigg[\begin{matrix}-a+z&1+\f{-a+z}{2}&-a+sp&-a+tp\\ &\f{-a+z}{2}&1+z-sp&1+z-tp\end{matrix}\bigg|\ 1\bigg]\\
=&\f{\Gamma(1-sp)}{\Gamma\l(\f{1+a-2s p}{2}\r)}\cdot\f{\Gamma(1-tp)}{\Gamma\l(\f{1+a-2tp}{2}\r)}\cdot\f{\Gamma\l(\f{1+3a-2s p-2t p}{2}\r)}{\Gamma(1+a-sp-tp)}\cdot\lim_{z\to0}\f{\Gamma\l(\f{1-a+z}{2}\r)}{\Gamma(1-a+z)}.
\end{align*}
Since $a$ is odd and $a<(2p+1)/3$, by \eqref{padicgamma1} we have
\begin{gather*}
\f{\Gamma(1-sp)}{\Gamma\l(\f{1+a-2s p}{2}\r)}=\prod_{j=1}^{\f{a-1}{2}}\f{1}{j+sp}=(-1)^{\f{a-1}{2}}\cdot\f{\Gamma_p(1-sp)}{\Gamma_p\l(\f{1+a-2s p}{2}\r)},\\
\f{\Gamma(1-tp)}{\Gamma\l(\f{1+a-2tp}{2}\r)}=\prod_{j=1}^{\f{a-1}{2}}\f{1}{j+tp}=(-1)^{\f{a-1}{2}}\cdot\f{\Gamma_p(1-tp)}{\Gamma_p\l(\f{1+a-2t p}{2}\r)},\\
\f{\Gamma\l(\f{1+3a-2s p-2t p}{2}\r)}{\Gamma(1+a-sp-tp)}=\prod_{j=1}^{\f{a-1}{2}}(a+j-sp-tp)=(-1)^{\f{a-1}{2}}\f{\Gamma_p\l(\f{1+3a-2s p-2t p}{2}\r)}{\Gamma_p(1+a-sp-tp)}.
\end{gather*}
In light of \eqref{eulerreflection} we obtain
\begin{gather*}
\Gamma\l(\f{1-a+z}{2}\r)\Gamma\l(\f{1+a-z}{2}\r)=\f{\pi}{\sin\pi\f{1-a+z}{2}},\\
\Gamma(1-a+z)\Gamma(a-z)=\f{\pi}{\sin\pi(1-a+z)}.
\end{gather*}
Furthermore,
\begin{align*}
\lim_{z\to0}\f{\Gamma\l(\f{1-a+z}{2}\r)}{\Gamma(1-a+z)}=&\lim_{z\to0}\f{\Gamma(a-z)}{\Gamma(\f{1+a-z}{2})}\cdot\f{\sin\pi(1-a+z)}{\sin\pi\f{1-a+z}{2}}\\
=&2(-1)^{\f{a-1}{2}}\f{\Gamma_p(a)}{\Gamma_p(\f{1+a}{2})}\cdot\f{\cos\pi(a-1)}{\cos\pi\f{a-1}{2}}=\f{2\Gamma_p(a)}{\Gamma_p(\f{1+a}{2})}.
\end{align*}
Noting \eqref{eulerpadic} we have
$$
\f{2\Gamma_p(a)}{\Gamma_p(\f{1+a}{2})}=(-1)^{\f{a-1}{2}}\f{2\Gamma_p\l(\f{1-a}{2}\r)}{\Gamma_p(1-a)}.
$$
By the above we arrive at
\begin{align*}
&{}_4F_3\bigg[\begin{matrix}-a&1+\f{-a}{2}&-a+sp&-a+tp\\ &\f{-a}{2}&1-sp&1-tp\end{matrix}\bigg|\ 1\bigg]_{p-1}\\
=&\f{\Gamma_p(1-sp)}{\Gamma_p\l(\f{1+a-2s p}{2}\r)}\cdot\f{\Gamma_p(1-tp)}{\Gamma_p\l(\f{1+a-2t p}{2}\r)}\cdot\f{\Gamma_p\l(\f{1+3a-2s p-2t p}{2}\r)}{\Gamma_p(1+a-sp-tp)}\cdot\f{2\Gamma_p\l(\f{1-a}{2}\r)}{\Gamma_p(1-a)}.
\end{align*}
Thus \eqref{r=0} is concluded.

Now we turn to show
\begin{equation}\label{s=0}
\Psi(rp,0,tp)=0
\end{equation}
for any $r,t\in\Z_p$.
Also, we may assume that $rp,tp,rp-tp\in\Q\setminus\Z$. By \eqref{4F3identity} we have
\begin{align*}
&{}_4F_3\bigg[\begin{matrix}-a+rp&1+\f{-a+rp}{2}&-a&-a+tp\\ &\f{-a+rp}{2}&1+rp&1+rp-tp\end{matrix}\bigg]_{p-1}\\
=&{}_4F_3\bigg[\begin{matrix}-a+rp&1+\f{-a+rp}{2}&-a&-a+tp\\ &\f{-a+rp}{2}&1+rp&1+rp-tp\end{matrix}\bigg]\\
=&\f{\Gamma(1+rp)}{\Gamma(1-a+rp)}\cdot\f{\Gamma(1+rp-tp)}{\Gamma(1+a+rp-tp)}\cdot\f{\Gamma\l(\f{1-a+rp}{2}\r)}{\Gamma\l(\f{1+a+rp}{2}\r)}\cdot\f{\Gamma\l(\f{1+3a+rp}{2}-tp\r)}{\Gamma\l(\f{1+a+rp}{2}-tp\r)}\\
=&\f{rp\cdot\Gamma_p(1+rp)}{\Gamma_p(1-a+rp)}\cdot\f{\Gamma_p(1+rp-tp)}{\Gamma_p(1+a+rp-tp)}\cdot\f{\Gamma_p\l(\f{1-a+rp}{2}\r)}{\f{1}{2}rp\cdot\Gamma_p\l(\f{1+a+rp}{2}\r)}\cdot\f{\Gamma_p\l(\f{1+3a+rp}{2}-tp\r)}{\Gamma_p\l(\f{1+a+rp}{2}-tp\r)}.
\end{align*}
Thus \eqref{s=0} holds. Symmetrically, we also have $\Psi(rp,sp,0)=0$ for any $r,s\in\Z_p$. The proof of the first case is now complete.\qed
\begin{remark}
Set
\begin{align*}
\Phi(x,y,z):=&{}_4F_3\bigg[\begin{matrix}-a+x&1+\f{-a+x}{2}&-a+y&-a+z\\ &\f{-a+x}{2}&1+x-y&1+x-z\end{matrix}\bigg|\ 1\bigg]_{p-1}\notag\\
&-f(x,y,z)\cdot\f{\Gamma_p(1+x-y)\Gamma_p(1+x-z)\Gamma_p\l(\f{1-a+x}{2}\r)\Gamma_p\l(\f{1+3a+x-2y-2z}{2}\r)}{\Gamma_p(1-a+x)\Gamma_p(1+a+x-y-z)\Gamma_p\l(\f{1+a+x-2y}{2}\r)\Gamma_p\l(\f{1+a+x-2z}{2}\r)},
\end{align*}
where
$$
f(x,y,z)=
\begin{cases}
\displaystyle 2p+x-2y-2z,\quad & \t{$a$ is odd and $a\geq\f{2p+1}{3}$,}\vspace{2mm}\\
\displaystyle x,\quad & \t{$a$ is even and $a<\f{p+1}{3}$,}\vspace{2mm}\\
\displaystyle x\cdot\f{p+x-2y-2z}{2},\quad & \t{$a$ is even and $a\geq\f{p+1}{3}$.}
\end{cases}
$$
To obtain the remaining cases we only need to show the following result whose proof is left to the reader as an excercise:
$$
\Phi(rp,sp,tp)=0
$$
provided $r,s,t\in\Z_p$ and at least one of $r,s,t$ is zero.
\end{remark}
\medskip
We now consider the $p$-adic analogue of the following identity due to Whipple.
\begin{lemma}\cite[(5.1)]{Whipple}\label{whipple}
For any $\alpha,\beta,\gamma\in\C$ we have
$$
{}_4F_3\bigg[\begin{matrix}\alpha&1+\f{\alpha}{2}&\beta&\gamma\\&\f{\alpha}{2}&1+\alpha-\beta&1+\alpha-\gamma\end{matrix}\bigg|-1\bigg]=\f{\Gamma(1+\alpha-\beta)\Gamma(1+\alpha-\gamma)}{\Gamma(1+\alpha)\Gamma(1+\alpha-\beta-\gamma)}.
$$
\end{lemma}
\medskip
\noindent{\it Proof of Theorem \ref{4F3(-1)analog}}. Let $a=\<-\alpha\>_p$. Set
\begin{align*}
\Omega(x,y,z):=&{}_4F_3\bigg[\begin{matrix}-a+x&1+\f{-a+x}{2}&-a+y&-a+z\\&\f{-a+x}{2}&1+x-y&1+x-z\end{matrix}\bigg|-1\bigg]_{p-1}\\
&-x\cdot\f{\Gamma_p(1+x-y)\Gamma_p(1+x-z)}{\Gamma_p(1-a+x)\Gamma_p(1+a+x-y-z)}.
\end{align*}
Similarly as in the proof of Theorem \ref{4F3analog}, it suffices to show that
\begin{equation}\label{key2}
\Omega(rp,sp,tp)\eq0\pmod{p^3}
\end{equation}
provided $r,s,t\in\Z_p$ and at least one of $r,s,t$ is zero. Again, we may assume $p\geq7$.

We first consider the case $r=0$. Also, assume that $sp,tp,(s+t)p\in\Q\setminus\Z$. By Lemma \ref{whipple}, we find that
\begin{align*}
&{}_4F_3\bigg[\begin{matrix}-a&1+\f{-a}{2}&-a+sp&-a+tp\\&\f{-a}{2}&1-sp&1-tp\end{matrix}\bigg|-1\bigg]_{p-1}={}_4F_3\bigg[\begin{matrix}-a&1+\f{-a}{2}&-a+sp&-a+tp\\&\f{-a}{2}&1-sp&1-tp\end{matrix}\bigg|-1\bigg]\\
=&\f{\Gamma(1-sp)\Gamma(1-tp)}{\Gamma(1-a)\Gamma(1+a-sp-tp)}.
\end{align*}
Since $\alpha\in\Z_p^{\times}$ we have $a-1\in\{1,\ldots,p-2\}$. By \eqref{eulerreflection} we know that for any nonnegative integer $n$, $1/\Gamma(-n)$=0. Thus we obtain $\Omega(0,sp,tp)=0$.

Below we consider the case $s=0$. Assume that $rp,(r-t)p\in\Q\setminus\Z$. With the help of Lemma \ref{whipple} we arrive at
\begin{align*}
&{}_4F_3\bigg[\begin{matrix}-a+rp&1+\f{-a+rp}{2}&-a&-a+tp\\&\f{-a+rp}{2}&1+rp&1+rp-tp\end{matrix}\bigg|-1\bigg]_{p-1}\\
=&{}_4F_3\bigg[\begin{matrix}-a+rp&1+\f{-a+rp}{2}&-a&-a+tp\\&\f{-a}{2}&1+rp&1+rp-tp\end{matrix}\bigg|-1\bigg]\\
=&\f{\Gamma(1+rp)\Gamma(1+rp-tp)}{\Gamma(1-a+rp)\Gamma(1+a+rp-tp)}.
\end{align*}
By \eqref{padicgamma1} we have
$$
\f{\Gamma(1+rp)}{\Gamma(1-a+rp)}=(-1)^arp\cdot\f{\Gamma_p(1+rp)}{\Gamma_p(1-a+rp)}
$$
and
$$
\f{\Gamma(1+rp-tp)}{\Gamma(1+a+rp-tp)}=(-1)^a\cdot\f{\Gamma_p(1+rp-tp)}{\Gamma_p(1+a+rp-tp)}.
$$
Thus we obtain $\Omega(rp,0,tp)=0$. Symmetrically, we have $\Omega(rp,sp,0)=0$.
Combining the above we get \eqref{key2}. The proof of Theorem \ref{4F3(-1)analog} is now complete.\qed
\medskip
\section{Applications of Theorems \ref{4F3analog} and \ref{4F3(-1)analog}}
\setcounter{lemma}{0}
\setcounter{theorem}{0}
\setcounter{equation}{0}
\setcounter{conjecture}{0}
\setcounter{remark}{0}
In \cite[Conjecture 35]{SunZW19}, Sun posed many interesting conjectures, such as \eqref{sunintval1} and \eqref{sunintval2}. These two congruences look very challenging and cannot be solved by the method in this paper. Fortunately, our method can be used to solve several other conjectures. Here we list the conjectures that we shall prove in this section.
\begin{conjecture}\cite[Conjecture 35]{SunZW19}\label{sunconj1}
(i) Let $p>3$ be a prime. For any $x\in\Z_p$ with $3x\not\eq1,2\pmod{p}$ we have
\begin{equation}\label{sunconj1_1}
\sum_{k=0}^{p-1}(-1)^k(2k+1)\sum_{j=0}^k\binom{-x}{j}^3\binom{x-1}{k-j}^3\eq0\pmod{p^2}.
\end{equation}
For any $x\in\Z_p$ with $x\eq1/3\pmod{p}$, we have
\begin{equation}\label{sunconj1_2}
\sum_{k=0}^{p-1}(-1)^k(2k+1)\sum_{j=0}^k\binom{-x}{j}^3\binom{x-1}{k-j}^3\eq x+\f{p\l(\f{p}{3}\r)-1}{3}\pmod{p^2}.
\end{equation}
(ii) Let $p$ be an odd prime. If $p\not\eq5\pmod{8}$, then
\begin{equation}\label{sunconj1_3}
\sum_{k=0}^{p-1}(-1)^k(2k+1)\sum_{j=0}^{k}\binom{-1/4}{j}^3\binom{-3/4}{k-j}^3\eq p^2\pmod{p^3}.
\end{equation}
If $p\eq5,7\pmod{8}$, then
\begin{equation}\label{sunconj1_4}
\sum_{k=0}^{p-1}(2k+1)\sum_{j=0}^{k}\binom{-1/2}{j}^3\binom{-1/2}{k-j}^3\eq0\pmod{p^3}.
\end{equation}
\end{conjecture}

\begin{theorem}\label{appl1}
Conjecture \ref{sunconj1} is true.
\end{theorem}
\begin{remark}
(a) In fact, we could evaluate the sums in \eqref{sunconj1_1} and \eqref{sunconj1_2} modulo $p^3$, however, we shall not list them here since the mod $p^3$ results are very complicated.\\
(b) Sun also conjectured that for any odd prime $p\eq2\pmod{3}$,
$$
\sum_{k=0}^{p-1}(2k+1)\sum_{j=0}^{k}\binom{-1/3}{j}^3\binom{-2/3}{k-j}^3\eq0\pmod{p^3}.
$$
Via a similar discussion as the one in the proof of Theorem \ref{appl1}, one may find that it suffices to evaluate
$$
\sum_{k=0}^{p-1}\binom{-1/3}{k}^3\quad \t{and}\quad \sum_{k=0}^{p-1}\binom{-2/3}{k}^3
$$
modulo $p^2$.
\end{remark}

To show Theorem \ref{appl1} we need the following known results.
\begin{lemma}\cite[Theorem 5.1]{PTW}\label{ptw3F21}
Suppose that $p$ is an odd prime and $\alpha\in\Z_p^\times$.
Let $s=(\alpha+\langle-\alpha\rangle_p)/p$, and
$$g_p(\alpha)=\frac{\Gamma_p(1+\frac12\alpha)\Gamma_p(1-\frac32\alpha)}{\Gamma_p(1+\alpha)\Gamma_p(1-\alpha)\Gamma_p(1-\frac12\alpha)^2}.$$
Then the following congruence holds modulo $p^3$,
$$
{}_3F_2\bigg[\begin{matrix}
\alpha&\alpha&\alpha\\
&1&1
\end{matrix}\bigg|\,1\bigg]_{p-1}\equiv
\begin{cases}
\displaystyle 2g_p(\alpha)&\text{if $\langle-\alpha\rangle_p$ is even and $\langle-\alpha\rangle_p<2p/3$,}
\vspace{2mm}\\
\displaystyle p(2-3s)g_p(\alpha)&\text{if $\langle-\alpha\rangle_p$ is even and $\langle-\alpha\rangle_p\geq 2p/3$,}
\vspace{2mm}\\
\displaystyle ps g_p(\alpha)&\text{if $\langle-\alpha\rangle_p$ is odd and $\langle-\alpha\rangle_p<p/3$,}
\vspace{2mm}\\
\displaystyle \frac{p^2s(1-3s)g_p(\alpha)}{2}&\text{if $\langle-\alpha\rangle_p$ is odd and $\langle-\alpha\rangle_p\geq p/3$.}
\end{cases}
$$
\end{lemma}
\begin{lemma}\cite[Corollary 8.1]{PTW}\label{ptw3F2-1}
Let $p$ be an odd prime. If $p\equiv 1,3\pmod{8}$, then
$$
{}_3F_2\bigg[\begin{matrix}\frac12&\frac12&\frac12\\ &1&1\end{matrix}\bigg| -1\bigg]_{p-1}\equiv -\Gamma_p\bigg(\frac18\bigg)^2\Gamma_p\bigg(\frac38\bigg)^2\pmod{p^3}.
$$
If prime $p\equiv 5,7\pmod{8}$, then
$$
{}_3F_2\bigg[\begin{matrix}\frac12&\frac12&\frac12\\ &1&1\end{matrix}\bigg| -1\bigg]_{p-1}\equiv \frac{3p^2}{64}\cdot\Gamma_p\bigg(\frac18\bigg)^2\Gamma_p\bigg(\frac38\bigg)^2\pmod{p^3}.
$$
\end{lemma}

\medskip
\noindent{\it Proof of Theorem \ref{appl1}}. Recall that in Section 1, we obtained that
\begin{equation}\label{sigma12}
\sum_{k=0}^{p-1}\varepsilon^k(2k+1)\sum_{j=0}^k\binom{-x}{j}^3\binom{x-1}{k-j}^3\eq(1-x)\Sigma_1+x\Sigma_2\pmod{p^3},
\end{equation}
where
$$
\Sigma_1:={}_3F_2\bigg[\begin{matrix}x&x&x\\ &1&1\end{matrix}\bigg|-\varepsilon\bigg]_{p-1}\cdot{}_4F_3\bigg[\begin{matrix}1-x&1+\f{1-x}{2}&1-x&1-x\\ &\f{1-x}{2}&1&1\end{matrix}\bigg|-\varepsilon\bigg]_{p-1}
$$
and
$$
\Sigma_2:={}_3F_2\bigg[\begin{matrix}1-x&1-x&1-x\\ &1&1\end{matrix}\bigg|-\varepsilon\bigg]_{p-1}\cdot{}_4F_3\bigg[\begin{matrix}x&1+\f{x}{2}&x&x\\ &\f{x}{2}&1&1\end{matrix}\bigg|-\varepsilon\bigg]_{p-1}
$$

(i) Denote $\<-x\>_p$ by $a$. We first consider \eqref{sunconj1_1}. Here we only prove the case that $a$ is even since the other case can be confirmed similarly. If $a$ is even, then $\<x-1\>_p=p-1-a$ is even. Obviously, \eqref{sunconj1_1} holds if $p\mid a$. So we assume that $3x\not\eq0,1,2\pmod{p}$. Below we divide this case into three subcases.

\medskip
\noindent{\it Case 1}. $(p+1)/3\leq a<2p/3$.\\
In this case, $p/3-1<\<x-1\>_p\leq(2p-4)/3$. Since $3x\not\eq1,2\pmod{p}$ and $p>3$, we know that $\<x-1\>_p\neq(p-2)/3,(p-1)/3,p/3$. Thus $(p+1)/3\leq\<x-1\>_p\leq(2p-4)/3$. Then \eqref{sunconj1_1} follows from Theorem \ref{4F3analog} and Lemma \ref{ptw3F21}.

\medskip
\noindent{\it Case 2}. $a<(p+1)/3$.\\
Now $\<x-1\>_p>(2p-4)/3$. Since $3x\not\eq0,1,2\pmod{p}$ we have $\<x-1\>_p\neq(2p-3)/3,(2p-2)/3,(2p-1)/3$, and then $\<x-1\>_p\geq2p/3$. Then \eqref{sunconj1_1} follows again.

\medskip
\noindent{\it Case 3}. $a\geq2p/3$.\\
Clearly, $\<x-1\>_p\leq p/3-1$. Then \eqref{sunconj1_1} follows from Theorem \ref{4F3analog} and Lemma \ref{ptw3F21} immediately.

\medskip
Now we turn to \eqref{sunconj1_2}. We first assume that $p\eq1\pmod{6}$. In this case $a=(p-1)/3$ is even and $\<x-1\>_p=p-1-a=(2p-2)/3$. Thus by Theorem \ref{4F3analog} and Lemma \ref{ptw3F21} we have
$$
(1-x)\Sigma_1\eq0\pmod{p^2}
$$
and
\begin{align*}
x\Sigma_2\eq& 2x(x+a)\cdot\f{\Gamma_p\l(1+\f{1-x}{2}\r)\Gamma_p\l(1-\f{3(1-x)}{2}\r)\Gamma_p\l(\f{1+x}{2}\r)\Gamma_p\l(\f{1-3x}{2}\r)}{\Gamma_p(2-x)\Gamma_p(x)\Gamma_p\l(1-\f{1-x}{2}\r)^2\Gamma_p(1+x)\Gamma_p(1-x)\Gamma_p\l(\f{1-x}{2}\r)^2}\\
=&(x+a)\Gamma_p\l(\f{1-x}{2}\r)\Gamma_p\l(\f{1+x}{2}\r)\Gamma_p\l(\f{3x-1}{2}\r)\Gamma_p\l(\f{3-3x}{2}\r)\\
=&(x+a)(-1)^{\f{2(p-1)}{3}}=x+a=x+\f{p-1}{3}.
\end{align*}
Thus \eqref{sunconj1_2} follows from the fact that $\l(\f{p}{3}\r)=1$ for $p\eq1\pmod{6}$.

Below we suppose that $p\eq5\pmod{6}$. Now $a=(2p-1)/3$ is odd and $\<x-1\>_p=p-1-a=(p-2)/3$. Also, by Theorem \ref{4F3analog} and Lemma \ref{ptw3F21} we obtain
$$
(1-x)\Sigma_1\eq0\pmod{p^2}
$$
and
\begin{align*}
x\Sigma_2\eq& 2x(1-x+\<x-1\>_p)\cdot\f{\Gamma_p\l(1+\f{1-x}{2}\r)\Gamma_p\l(1-\f{3(1-x)}{2}\r)\Gamma_p\l(\f{1+x}{2}\r)\Gamma_p\l(\f{1-3x}{2}\r)}{\Gamma_p(2-x)\Gamma_p(x)\Gamma_p\l(1-\f{1-x}{2}\r)^2\Gamma_p(1+x)\Gamma_p(1-x)\Gamma_p\l(\f{1-x}{2}\r)^2}\\
=&(p-a-x)\Gamma_p\l(\f{1-x}{2}\r)\Gamma_p\l(\f{1+x}{2}\r)\Gamma_p\l(\f{3x-1}{2}\r)\Gamma_p\l(\f{3-3x}{2}\r)\\
=&(p-a-x)(-1)^{\f{p-2}{3}}=x+\f{-p-1}{3}.
\end{align*}
Noting that $\l(\f{p}{3}\r)=-1$ provided $p\eq5\pmod{6}$, \eqref{sunconj1_2} is concluded.

\medskip
(ii) We now consider \eqref{sunconj1_3}. We just prove the case $p\eq1\pmod{8}$ since the remaining cases are very similar. In this case, by Theorem \ref{4F3analog} and Lemma \ref{ptw3F21} we obtain that
\begin{gather*}
{}_3F_2\bigg[\begin{matrix}\f14&\f14&\f14\\ &1&1\end{matrix}\bigg|\ 1\bigg]_{p-1}\eq\f{2\Gamma_p\l(\f98\r)\Gamma_p\l(\f58\r)}{\Gamma_p\l(\f54\r)\Gamma_p\l(\f34\r)\Gamma_p\l(\f78\r)^2},\quad{}_3F_2\bigg[\begin{matrix}\f34&\f34&\f34\\ &1&1\end{matrix}\bigg|\ 1\bigg]_{p-1}\eq\f{-\f{p}{4}\cdot\Gamma_p\l(\f{11}8\r)\Gamma_p\l(-\f18\r)}{\Gamma_p\l(\f74\r)\Gamma_p\l(\f14\r)\Gamma_p\l(\f58\r)^2},\\
{}_4F_3\bigg[\begin{matrix}\f14&\f98&\f14&\f14\\ &\f18&1&1\end{matrix}\bigg|\ 1\bigg]_{p-1}\eq\f{p}{4}\cdot\f{\Gamma_p\l(\f58\r)\Gamma_p\l(\f18\r)}{\Gamma_p\l(\f54\r)\Gamma_p\l(\f34\r)\Gamma_p\l(\f38\r)^2},\\
{}_4F_3\bigg[\begin{matrix}\f34&\f{11}8&\f34&\f34\\ &\f38&1&1\end{matrix}\bigg|\ 1\bigg]_{p-1}\eq-\f{15p^2}{16}\cdot\f{\Gamma_p\l(\f78\r)\Gamma_p\l(-\f58\r)}{\Gamma_p\l(\f74\r)\Gamma_p\l(\f14\r)\Gamma_p\l(\f18\r)^2}.
\end{gather*}
Then by \eqref{eulerpadic} and \eqref{sigma12} we deduce that
$$
\sum_{k=0}^{p-1}(-1)^k(2k+1)\sum_{j=0}^{k}\binom{-1/4}{j}^3\binom{-3/4}{k-j}^3\eq\f{3}{4}\cdot p^2+\f{1}{4}\cdot p^2=p^2\pmod{p^3}.
$$
This proves \eqref{sunconj1_3}.

Below we turn to show \eqref{sunconj1_4}. By \eqref{sigma12} we have
\begin{align*}
&\sum_{k=0}^{p-1}(2k+1)\sum_{j=0}^{k}\binom{-1/2}{j}^3\binom{-1/2}{k-j}^3\\
\eq&{}_3F_2\bigg[\begin{matrix}\frac12&\frac12&\frac12\\ &1&1\end{matrix}\bigg| -1\bigg]_{p-1}\cdot{}_4F_3\bigg[\begin{matrix}\frac12&\f54&\frac12&\frac12\\ &\f14&1&1\end{matrix}\bigg| -1\bigg]_{p-1}\pmod{p^3}.
\end{align*}
By Theorem \ref{4F3(-1)analog} we know that
$$
{}_4F_3\bigg[\begin{matrix}\frac12&\f54&\frac12&\frac12\\ &\f14&1&1\end{matrix}\bigg| -1\bigg]_{p-1}\eq0\pmod{p}.
$$
In view of Lemma \ref{ptw3F2-1}, for $p\eq5,7\pmod{8}$ we have
$$
{}_3F_2\bigg[\begin{matrix}\frac12&\frac12&\frac12\\ &1&1\end{matrix}\bigg| -1\bigg]_{p-1}\eq0\pmod{p^2}.
$$
Combining the above, \eqref{sunconj1_4} follows immediately.

The proof of Theorem \ref{appl1} is now complete.\qed

\medskip
\section{Some more solvable conjectures}
\setcounter{lemma}{0}
\setcounter{theorem}{0}
\setcounter{equation}{0}
\setcounter{conjecture}{0}
\setcounter{remark}{0}
In the previous two sections, we established $p$-adic analogues of two hypergeometric identities and used them to solve some congruences conjectured by Sun. In fact, in \cite{SunZW19}, there are some more conjectures can be solved by some known results. Here we give a collection of them.
\begin{conjecture}\cite[Conjecture 36]{SunZW19}\label{sunconj}
{\rm (i)} For each prime $p>3$, we have
\begin{equation}\label{sunconj2_1}
\sum_{k=0}^{p-1}(2k+1)\sum_{j=0}^{k}\binom{-1/6}{j}^4\binom{-5/6}{k-j}^4\eq0\pmod{p^2}.
\end{equation}

{\rm (ii)} Let $p$ be an odd prime. If $p\eq3\pmod{4}$, then
\begin{equation}\label{sunconj2_2}
\sum_{k=0}^{p-1}(2k+1)\sum_{j=0}^{k}\binom{-1/4}{j}^4\binom{-3/4}{k-j}^4\eq0\pmod{p^2},
\end{equation}
and
\begin{equation}\label{sunconj2_3}
\sum_{k=0}^{p-1}(2k+1)\sum_{j=0}^{k}\binom{-1/2}{j}^5\binom{-1/2}{k-j}^5\eq0\pmod{p^3}.
\end{equation}
If $p\eq5\pmod{6}$, then
\begin{equation}\label{sunconj2_4}
\sum_{k=0}^{p-1}(2k+1)\sum_{j=0}^{k}\binom{-1/6}{j}^6\binom{-5/6}{k-j}^6\eq0\pmod{p^2}.
\end{equation}
\end{conjecture}

\begin{theorem}
Conjecture \ref{sunconj} is true.
\end{theorem}
\begin{proof}
We first illustrate that \eqref{sunconj2_3} holds. By \eqref{sigma1sigma2}, we know that
\begin{align*}
&\sum_{k=0}^{p-1}(2k+1)\sum_{j=0}^{k}\binom{-1/2}{j}^5\binom{-1/2}{k-j}^5\\
\eq&{}_6F_5\bigg[\begin{matrix}\f12&\f54&\f12&\f12&\f12&\f12\\&\f14&1&1&1&1\end{matrix}\bigg|-1\bigg]_{p-1}\cdot{}_5F_4\bigg[\begin{matrix}\f12&\f12&\f12&\f12&\f12\\&1&1&1&1\end{matrix}\bigg|-1\bigg]_{p-1}\pmod{p^5}.
\end{align*}
In \cite{vH}, van Hamme ever conjectured that for any odd prime $p\eq3\pmod{4}$
$$
{}_6F_5\bigg[\begin{matrix}\f12&\f54&\f12&\f12&\f12&\f12\\&\f14&1&1&1&1\end{matrix}\bigg|-1\bigg]_{p-1}\eq0\pmod{p^3}.
$$
This was confirmed by Liu in \cite{Liu19} by establishing its generalization: For $p\geq5$ with $p\eq3\pmod{4}$
$$
{}_6F_5\bigg[\begin{matrix}\f12&\f54&\f12&\f12&\f12&\f12\\&\f14&1&1&1&1\end{matrix}\bigg|-1\bigg]_{p-1}\eq-\f{p^3}{16}\Gamma_p\l(\f14\r)\pmod{p^4}.
$$
Note that Liu conjectured that the above congruence also holds modulo $p^5$ which has been confirmed by Wang \cite{Wang}. In view of these, \eqref{sunconj2_3} holds evidently.

Now we prove \eqref{sunconj2_1}. Assume $p\eq1\pmod{6}$. It follows that $\<-1/6\>_p=(p-1)/6<p/2$ and $\<-5/6\>_p=(5p-5)/6>p/2$. Now by \cite[Theorem 2.22]{MP}, we obtain that
$$
{}_5F_4\bigg[\begin{matrix}\f16&\f{13}{12}&\f16&\f16&\f16\\&\f1{12}&1&1&1\end{matrix}\bigg|\ 1\bigg]_{p-1}\eq0\pmod{p}\quad\t{and}\quad {}_5F_4\bigg[\begin{matrix}\f56&\f{17}{12}&\f56&\f56&\f56\\&\f5{12}&1&1&1\end{matrix}\bigg|\ 1\bigg]_{p-1}\eq0\pmod{p^2}.
$$
By \eqref{sigma1sigma2}, it suffices to show that
$$
{}_4F_3\bigg[\begin{matrix}\f56&\f56&\f56&\f56\\&1&1&1\end{matrix}\bigg|\ 1\bigg]_{p-1}\eq0\pmod{p}.
$$
We need the well-known Karlsson-Minton summation formula (cf. \cite[Page 18]{GR}):
\begin{equation}\label{KMidentity}
{}_{r+1}F_r\bigg[\begin{matrix}a&b_1+m_1&\cdots&b_r+m_r\\&b_1&\cdots&b_r\end{matrix}\bigg|\ 1\bigg]=0,
\end{equation}
provided that $m_1,m_2,\ldots,m_r$ are nonnegative integers and $\Re(-a)>m_1+\cdots+m_r$. By \eqref{KMidentity}, we have
$$
{}_4F_3\bigg[\begin{matrix}\f56&\f56&\f56&\f56\\&1&1&1\end{matrix}\bigg|\ 1\bigg]_{p-1}\eq{}_4F_3\bigg[\begin{matrix}\f{5-5p}6&\f{5+p}6&\f{5+p}6&\f{5+p}6\\&1&1&1\end{matrix}\bigg|\ 1\bigg]=0\pmod{p}
$$
since $(p-1)/2<(5p-5)/6$.

If $p\eq5\pmod{6}$, then we have $\<-1/6\>_p=(5p-1)/6>p/2$ and $\<-5/6\>_p=(p-5)/6<p/2$. Thus by \cite[Theorem 2.22]{MP} we arrive at
$$
{}_5F_4\bigg[\begin{matrix}\f16&\f{13}{12}&\f16&\f16&\f16\\&\f1{12}&1&1&1\end{matrix}\bigg|\ 1\bigg]_{p-1}\eq0\pmod{p^2}\quad\t{and}\quad {}_5F_4\bigg[\begin{matrix}\f56&\f{17}{12}&\f56&\f56&\f56\\&\f5{12}&1&1&1\end{matrix}\bigg|\ 1\bigg]_{p-1}\eq0\pmod{p}.
$$
Furthermore, by \eqref{KMidentity} we have
$$
{}_4F_3\bigg[\begin{matrix}\f16&\f16&\f16&\f16\\&1&1&1\end{matrix}\bigg|\ 1\bigg]_{p-1}\eq{}_4F_3\bigg[\begin{matrix}\f{1-5p}6&\f{1+p}6&\f{1+p}6&\f{1+p}6\\&1&1&1\end{matrix}\bigg|\ 1\bigg]=0\pmod{p}
$$
since $(p-5)/2<(5p-1)/6$. Combining the above \eqref{sunconj2_1} holds.

\eqref{sunconj2_2} can be verified in a similar way as the one in the proof of \eqref{sunconj2_1}. We now consider \eqref{sunconj2_4}. Since $p\eq5\pmod{6}$, we have $\<-1/6\>_p=(5p-1)/5>2p/3$ and $\<-5/6\>_p=(p-5)/6<p/3$. Then by \cite[Theorems 2.17 and 2.20]{MP}, we obtain
$$
{}_7F_6\bigg[\begin{matrix}\f16&\f{13}{12}&\f16&\f16&\f16&\f16&\f16\\&\f1{12}&1&1&1&1&1\end{matrix}\bigg|\ 1\bigg]_{p-1}\eq0\pmod{p^2}
$$
and
$$
{}_7F_6\bigg[\begin{matrix}\f56&\f{17}{12}&\f56&\f56&\f56&\f56&\f56\\&\f5{12}&1&1&1&1&1\end{matrix}\bigg|\ 1\bigg]_{p-1}\eq0\pmod{p}.
$$
Similarly to the above, by \eqref{KMidentity} we may easily obtain that
$$
{}_6F_5\bigg[\begin{matrix}\f16&\f16&\f16&\f16&\f16&\f16\\&1&1&1&1&1\end{matrix}\bigg|\ 1\bigg]_{p-1}\eq0\pmod{p}.
$$
By \eqref{sigma1sigma2} we immediately obtain \eqref{sunconj2_4}.

Now the proof of Theorem \ref{sunconj} is complete.
\end{proof}

\end{document}